\documentclass{amsart}

\usepackage{amsfonts,amsmath,amssymb,amstext,amsthm}
\usepackage{epic}
\usepackage{graphicx,color}
\usepackage{mathrsfs}

\newcommand{\nwc}{\newcommand}

\newcommand{\proje}{\mathbb{P}}
\newcommand{\cal}{\mathcal}

\nwc{\aaa}{\mathcal{F}}
\nwc{\aab}{\bar{\mathfrak{a}}}
\nwc{\aal}{\mathcal{F}'}
\nwc{\aap}{\mathcal{F}_{P}}

\nwc{\bbb}{\mathfrak{b}}
\nwc{\bbp}{\mathfrak{b}_{P}}
\nwc{\be}{\begin{equation}}
\nwc{\bea}{\begin{eqnarray}}
\nwc{\beq}{$$}

\nwc{\C}{\;\mbox{{\sf I}}\!\!\!C}
\nwc{\cb}{\overline{C}}
\nwc{\ccc}{\mathcal{C}}
\nwc{\cin}{\textbf{(v)}}
\nwc{\cl}{C'}
\nwc{\cp}{\mathcal{C}_{P}}
\nwc{\cpll}{\mathcal{C}_{P'}}
\nwc{\ct}{\tilde{C}}

\nwc{\dd}{\mathcal{L}}
\nwc{\ddd}{\mathfrak{d}}
\nwc{\ddl}{\mathcal{L}'}
\nwc{\dlp}{\delta_{P}}
\nwc{\doi}{\textbf{(ii)}}

\nwc{\ee}{\end{equation}}
\nwc{\eea}{\end{eqnarray}}
\nwc{\enq}{$$}

\nwc{\G}{{\cal G}}
\nwc{\gon}{{\rm gon}}
\nwc{\gtl}{\tilde{g}}
\nwc{\gud}{g^{1}_{2}}
\nwc{\gtu}{g^{1}_{3}}

\nwc{\hhza}{H^{0}(C,\mathfrak{a})}
\nwc{\hua}{h^{1}(C,\mathfrak{a})}
\nwc{\hza}{h^{0}(C,\mathfrak{a})}

\nwc{\kk}{{\rm K}}

\nwc{\lbd}{\lambda}
\nwc{\lif}{L_{\infty}}

\nwc{\mm}{\mathfrak{m}}
\nwc{\mmp}{\mathfrak{m}_{P}}
\nwc{\mpd}{{\mathfrak{m}_{P}}^{2}}

\nwc{\N}{I\!\!N}
\nwc{\nn}{\mathbb{N}}

\nwc{\obp}{\overline{\mathcal{O}}_P}
\nwc{\ocbux}{\oo _{\bar{C}}\langle 1,x\rangle}
\nwc{\oclux}{\oo _{C'}\langle 1,x\rangle}
\nwc{\ocux}{\oo _{C}\langle 1,x\rangle}
\nwc{\ol}{\mathcal{O}'}
\nwc{\oma}{\Omega (\mathfrak{a})}
\nwc{\omo}{\Omega (\mathcal{O})}
\nwc{\oo}{\mathcal{O}}
\nwc{\ooh}{\widehat{\mathcal{O}}}
\nwc{\op}{\mathcal{O}_P}
\nwc{\opc}{\mathcal{O}_{P,C}}
\nwc{\oph}{\widehat{\mathcal{O}}_{P}}
\nwc{\opl}{\mathcal{O}_{P}'}
\nwc{\oplc}{\mathcal{O}_{P,C}'}
\nwc{\opll}{\mathcal{O}_{P'}}
\nwc{\opt}{\tilde{\mathcal{O}}_{P}}
\nwc{\optt}{{\mathcal{O}}_{\tilde{P}}}
\nwc{\oq}{\mathcal{O}_{Q}}
\nwc{\oqt}{\tilde{\mathcal{O}}_{Q}}
\nwc{\ot}{\tilde{\mathcal{O}}}
\nwc{\overop}{\bar{\oo}_{P}}

\nwc{\pb}{\overline{P}}
\nwc{\pgmd}{\mathbb{P}^{g+2}}
\nwc{\pgmu}{\mathbb{P}^{g+1}}
\nwc{\pp}{\mathbb{P}}
\nwc{\prv}{\noindent\textbf{Proof}:}
\nwc{\pt}{\tilde{P}}
\nwc{\ptl}{\tilde{P}}
\nwc{\pum}{\mathbb{P}^{1}}
\nwc{\carta}{\mathfrak{U}}

\nwc{\Q}{\;\mbox{{\sf I}}\!\!\!Q}
\nwc{\qtl}{\tilde{Q}}
\nwc{\qua}{\textbf{(iv)}}

\nwc{\R}{I\!\!R}

\nwc{\sep}{\beq\ast\ \ast\ \ast\enq}
\nwc{\spl}{{S_{P}}'}
\nwc{\spll}{S_{P'}}
\nwc{\ssp}{{\rm S}_{P}}
\nwc{\sss}{{\rm S}}
\nwc{\sys}{\mathcal{L}}

\nwc{\tre}{\textbf{(iii)}}

\nwc{\um}{\textbf{(i)}}

\nwc{\vlp}{\mathcal{V}_{\lambda,P}}
\nwc{\vpt}{v_{\ptl}}
\nwc{\vv}{\mathcal{W}}
\nwc{\vvp}{\mathcal{W}_{P}}

\nwc{\wol}{\ww\cdot\mathcal{O}'}
\nwc{\wpn}{{\omega _{P}}^{n}}
\nwc{\ww}{\omega}
\nwc{\wwp}{\omega _{P}}

\nwc{\Z}{{Z\!\!\!Z}}
\nwc{\zz}{\mathbb{Z}}
\newcommand{\Pic}{\operatorname{Pic}}

\newtheorem{coro}{Corollary}[section]
\newtheorem{defi}[coro]{Definition}
\newtheorem{exem}[coro]{Example}

\newtheorem{lema}[coro]{Lemma}

\newtheorem{rem}[coro]{Remark}
\newtheorem{teo}[coro]{Theorem}

\begin{document}

\title{On Gonality, Scrolls, and Canonical Models of Non-Gorenstein Curves}

\author{Danielle Lara}
\address{Departamento de Matem\'atica, UFV / CAF,
Rodovia LMG 818 km 06,
35690-000 Florestal MG, Brazil}
\email{danilara@ufv.br}

\author{Jairo Menezes Souza}
\address{Unidade Acad\^emica Especial de Matem\'atica e Tecnologia - IMTec/RC/UFG
Av. Dr. Lamartine Pinto de Avelar 1120, St. Universitário,
75704-020,  Catalão - Goiás - Brasil.}
\email{jairoms@ufg.br}

\author{Renato Vidal Martins}
\address{Departamento de Matem\'atica, Instituto de Ci\^encias Exatas, UFMG,
Av. Ant\^onio Carlos 6627,
30123-970 Belo Horizonte MG, Brazil}
\email{renato@mat.ufmg.br}

\subjclass{Primary 14H20, 14H45, 14H51}

\keywords{non-Gorenstein curve, canonical model, gonality, scrolls}

\begin{abstract}
Let  $C$ be an integral and projective curve; and let $C'$ be its canonical model. We study the relation between the gonality of $C$ and the dimension of a rational normal scroll $S$ where $C'$ can lie on. We are mainly interested in the case where $C$ is singular, or even non-Gorenstein, in which case $C'\not\cong C$. We first analyze some properties of an inclusion $C'\subset S$ when it is induced by a pencil on $C$. Afterwards, in an opposite direction, we assume $C'$ lies on a certain scroll, and check some properties $C$ may satisfy, such as gonality and the kind of its singularities. At the end, we prove that a rational monomial curve $C$ has gonality $d$ if and only if $C'$ lies on a $(d-1)$-fold scroll.
\end{abstract}

\maketitle


\section*{Introduction}

 F. Enriques and D. W. Babbage \cite{En,B} proved that a nonhyperelliptic smooth canonical curve is the set theoretic intersection of hyperquadrics, unless it is trigonal or isomorphic to a plane quintic. It can be read off from their works the following statement: ``a regular, integral and projective curve is trigonal if and only if it is isomorphic to a canonical curve which lies on a nonsingular two-dimensional rational normal scroll"

From this perspective, two generalizations of such a characterization are quite natural, and, in fact, were actually done. One may consider higher gonality, or, rather, allow trigonal curves to have singularities.

In the first vein above, one finds, for instance, in Schreyer's \cite[Sec. 6]{Sc}, a detailed study of the relation between a $d$-gonal canonical curve $C$ and the $(d-1)$-fold scroll $S$ it lies on, specially when $d=4,5$ (and $d=3$ as well). This envolves, e.g., verifying the uniqueness of the $g_d^1$; finding the resolution of $C$ inside $S$, and also determining upper and lower bounds for the invariants of $S$ in terms of the genus of $C$.

On the other way, in \cite{RS}, St\"ohr and Rosa devoted their study to the case where both the trigonal curve and the surface scroll are singular. Their results perfectly match the statement above replacing ``regular" by ``Gorenstein".  A key point in their approach was that the linear series could possibly admit non-removable base points. This relaxation on the standard notion of a pencil turned out to be necessary once they proved that a canonical (Gorenstein) curve lying on a cone always meets the vertex, which cannot be removed, otherwise the curve would be hyperelliptic. Actually, systems with non-removable base points appear earlier in the literature, introduced by Coppens in \cite{Cp}. Essentially, one is allowing torsion free sheaves of rank $1$ (rather than bundles) in the definition of linear series,
which, by its turn, allows pencils (rather than morphisms) to compute gonality (see Section \ref{secs21}).

When this study is led to non-Gorenstein curves, soon we come accross the following question: what exactly would be a generalization of the statement above in such a case? The terms ``isomorphic" and ``canonical"  leads now to different ways, depending on which choice we make. Indeed, given a $d$-gonal non-Gorenstein curve $C$, one may look for a $(d-1)$-fold scroll containing: (i) either an isomorphic copy of $C$ by means of a suitable embedding; (ii) or, rather, a curve $C'$ which could be naturally called a ``canonical model" for $C$. 

In Example \ref{exedif} we make few remarks about the difficulties of getting (i) above since the standard methods of inducing inclusions on scrolls by pencils may fail to reach the expected dimension when $C$ is not Gorenstein. On the other hand, if (ii) stands for an option, one may deal, for instance, with the notion of a \emph{canonical model} $C'$ introduced by Rosenlicht in \cite{R} and also studied by S. L. Kleiman along with the third named author in \cite{KM} (see Section \ref{seccan}). Moreover, within this framework, \cite[Thm. 3.4]{KM} which states that  $C'$ is the (non-degenerate) rational normal curve (of minimal degree) if and only if $C$ is either hyperelliptic or rational nearly normal, 
can be rephrased as: $C$ is $2$-gonal if and only if $C'$ lies on a $1$-fold scroll. So the mere formalism of considering rational normal curves as scrolls turns here into the first step of a general result, proved in Theorem \ref{thmrat} for monomial curves. 

We start discussing the general case in Section 2. We show in Theorem \ref{thmscr} how to get an inclusion $C'\subset S$ from any pencil on $C$. In particular we get that $S$ is $(d-1)$-dimensional if $C$ is $d$-gonal. The result extends St\"ohr-Rosa's \cite[Thm. 2.1, Lem 2.3]{RS} from trigonal curves to any gonality, using similar methods based on Andreotti-Mayer's \cite{AM}. We also give an upper bound for the dimension of the singular set of $S$ in terms of some invariants of the pencil, and look for sufficient conditions for $S$ to be in fact singular.

In Section 3 we do the reverse engeneering, that is, we assume $C'$ lies on a given scroll $S$ with prescribed dimension $d$ and  intersection number $\ell$ with a generic fiber of $S$. Varrying $\ell$, we are able to relate some important properties of $C$ with $d$ and other invariants of $S$. Our main concern is gonality, but we also study the number of non-Gorenstein points of $C$ and then check when the curve happens to be Kunz, nearly Gorenstein or nearly normal. These concepts, based on local principles introduced by Barucci and Fr\"oberg in \cite{BF},  got a geometric characterization in \cite[Thms. 5.10, 6.5]{KM} where they were connected to projective and arithmetic normality of the canonical model. They seem to be an essential tool when making first distinctions among non-Gorenstein curves. We summarize the results we got in Theorem \ref{thmth2}, which is a generalization to arbitrary $d$ of \cite[Thms. 2.1, 4.1]{LMM}, proved to $d=2,3$ by Marchesi along with the first and last named author. We close the section with Theorem \ref{thmth3}, which deals with a particular case.

However, we do not obtain a converse for the assertion of prior section. In fact, it will be clear the difficulty of adjusting the arguments of, for instance,  \cite{Sc, BS, CE}, even for small $d$, when the dualizing sheaf fails to be a bundle. On the other hand, in Section 4, we prove that $C$ is $d$-gonal if and only if $C'$ lies on a $(d-1)$-fold scroll if $C$ is a rational monomial curve in Theorem \ref{thmrat}. It generalizes \cite[Thms. 3.3, 5.1]{LMM} which was proved assuming $d=2,3$ and $C$ with just one singular point. The key point is a combinatorial description of $C'$ in Theorem \ref{prprat}, which is an extension of \cite[Prp 3.1]{LMM} and by means of slightly different arguments.

\

\noindent{\bf Acknowledgments.} The third named author is partially supported by CNPq grant number 306914/2015-8.


\section{Preliminaries}

For the remainder, a \emph{curve} is an integral and complete one-dimensional scheme over an algebraically closed ground field. Let $C$ be a curve of (arithmetic) genus $g$ with structure sheaf $\oo_C$, or simply $\oo$, and $k(C)$ the field of rational functions. Let $\pi :\overline{C}\rightarrow C$ be the normalization map, set $\overline{\oo}:=\pi _{*}(\oo _{\overline{C}})$ and call $\ccc:=\mathcal{H}\text{om}(\overline{\oo},\oo)$, the conductor of $\overline{\oo}$ into $\oo$. Let also $\ww_C$, or simply $\ww$, denote the dualizing sheaf of $C$. A point $P\in C$ is said to be \emph{Gorenstein} if $\ww_P$ is a free $\oo_P$-module. The curve is said to be \emph{Gorenstein} if all of its points are so, or equivalently, $\ww$ is invertible. It is said to be \emph{hyperelliptic} if there is a morphism $C\rightarrow\pp^1$ of degree $2$.

\subsection{Linear Systems and Gonality}
\label{secs21}
 
A \emph{linear system of dimension $r$ in $C$} is a set of the form
$$
{\rm L}:=\dd(\aaa ,V):=\{x^{-1}\aaa\ |\ x\in V\setminus 0\}
$$
where $\mathcal{F}$ is a coherent fractional ideal sheaf on $C$ and $V$ is a vector subspace of $H^{0}(\aaa )$ of dimension $r+1$. The sheaf above is given by 
$$
(x^{-1}\aaa)(U):=x^{-1}(\aaa(U))
$$
for any open set $U$ on $C$, which makes sense since $x^{-1}\in k(C)$ and $\aaa(U)\subset k(C)$. In other words, we are using the language of sheaves to the approach of ``divisors by product" of \cite{S}. That is,  linears systems are regularly defined by 
$$
\{D+\text{div}(x)\,|\, x\in V\subset L(D)\}
$$ 
for a (Weil) divisor $D$. On the other hand, here, ``divisors" are ``fractional ideal sheaves" so ``product" plays the role of ``sum" and ``inclusion" the role of ``inequality" (see \cite{S} for more details). So as $x\in L(D)$ if and only if $0\leq D+\text{div}(x)$, similarly, $x\in H^0(\aaa)$ if and only if $x\op\in\aaa(U)$ if and only if $\op\subset x^{-1}\aaa(U)$ for any open set $U$.  So the structure sheaf $\oo$ is $0$ and $x^{-1}\oo$ is $\text{div}(x)$ compared to the divisors theory. This approach happens to be useful specially for describing linear systems with non-removable base points, which we define below. The \emph{degree} of the linear system is the integer
$$
d:=\deg \aaa :=\chi (\aaa )-\chi (\oo)
$$
Note, in particular, that if $\oo\subset\aaa$ then
$$
\deg\aaa=\sum_{P\in C}\dim(\aaa_P/\op).
$$
The notation $g_{d}^{r}$ stands for ``linear system of degree $d$ and dimension $r$". The linear system is said to be \emph{complete} if $V=H^0(\aaa)$, in this case one simply writes ${\rm L}=|\aaa|$. The \emph{gonality} of $C$ is the
smallest $d$ for which there exists a $g_{d}^{1}$ in $C$, or equivalently, the smallest $d$ for which there exists a torsion free sheaf $\aaa$ of rank $1$ on $C$ with degree $d$ and $h^0(\aaa)\geq 2$. A point $P\in C$ is called a \emph{base point of ${\rm L}$} if $x\op\subsetneq\aaa_P$ for every $x\in V$. A base point is called \emph{removable} if it is not a base point of $\dd(\oo\langle V\rangle,V)$, where $\oo\langle V\rangle$ is the subsheaf of the constant sheaf of rational functions generated by all sections in $V\subset k(C)$. So $P$ is a non-removable base point of ${\rm L}$ if and only if $\aaa_P$ is not a free $\op$-module; in particular, $P$ is singular if so.

\subsection{The Canonical Model}
\label{seccan}

Given any integral scheme $A$, any map $\varphi :A\to C$ and a sheaf $\mathcal{G}$ on $C$, set
$$\mathcal{O}_A\mathcal{G}:= \varphi^*\mathcal G/\text{Torsion}(\alpha^*\mathcal G).$$

Given any coherent sheaf $\mathcal{F}$ on $C$ set $\mathcal{F}^n:=\text{Sym}^n\mathcal{F}/\text{Torsion}(\text{Sym}^n\mathcal{F})$. If $\mathcal{F}$ is invertible then clearly $\mathcal{F}^{n}=\mathcal{F}^{\otimes n}$.

Call $\widehat{C}:=\text{Proj}(\oplus\,\ww ^n)$ the blowup of $C$ along $\ww$. If $\widehat{\pi} :\widehat{C}\rightarrow C$ is the natural morphism, set $\widehat{\oo}=\widehat{\pi}_*(\oo_{\widehat{C}})$ and $\widehat{\oo}\ww:=\widehat{\pi}_*(\oo _{\widehat{C}}\ww)$. In \cite[p\,188\,top]{R} Rosenlicht showed that the linear system $\sys(\oo_{\overline{C}}\ww,H^0(\ww))$ is base point free. He considered then the morphism $\kappa :\overline{C}\rightarrow\pp^{g-1}$ induced by it and called $C':=\kappa(C)$ the \emph{canonical model} of $C$. He also proved in \cite[Thm\,17]{R} that if $C$ is nonhyperelliptic, the map
$\pi :\overline{C}\rightarrow C$ factors through a map $\pi' : C'\rightarrow C$. So set $\oo':=\pi'_*(\oo_{C'})$ in this case. In \cite[Dfn\,4.9]{KM} one finds another characterization of $C'$. It is the image of the morphism $\widehat{\kappa}:\widehat{C}\rightarrow\pp^{g-1}$ defined by the linear system $\sys(\oo_{\widehat{C}}\ww,H^0(\ww))$. By Rosenlicht's Theorem, since $\ww$ is generated by global sections, we have that $\widehat{\kappa}:\widehat{C}\rightarrow C'$ is an isomorphism if $C$ is nonhyperelliptic.

Now set $\overline{\oo}\ww:=\pi_*(\oo_{\overline{C}}\ww)$ and take $\lambda\in H^0(\ww)$ such that $(\overline{\oo}\ww)_P=\overline{\oo}_P\lambda$ for every singular point $P\in C$. Such a differential exists because $H^0(\ww)$ generates $\overline{\oo}\ww$ as proved in \cite[p\,188 top]{R}, and because the singular points of $C$ are of finite number and $k$ is infinite since it is algebraically closed. Set
\begin{equation}
\label{equvvv}
\vv=\vv_{\lambda}:=\ww/\lambda
\end{equation}
If so, we have
$$
\ccc_P\subset
\mathcal{O}_P \subset \vvp\subset\oph=\op'\subset\obp
$$
for every singular point $P\in C$, where the equality makes sense if and only if $C$ is nonhyperelliptic.

\begin{defi} \label{defnng}
\emph{Let $P\in C$ be any point. Set
$$
\eta_P:=\dim(\vvp/\op)\ \ \ \ \ \ \ \ \ \ \ \mu_P:=\dim({\widehat{\oo}_{P}}/\vvp)
$$
and also
$$
\eta:=\sum_{P\in C}\eta_P\ \ \ \ \ \ \ \ \ \ \mu:=\sum_{P\in C}\mu_P
$$
In particular, letting $g'$ be the genus of $C'$, we have
\begin{equation}
\label{equget}
g=g'+\eta+\mu
\end{equation}
Following \cite[Prps. 20, 21, 28]{BF}, we call $P$ \emph{Kunz} if $\eta_P=1$ and,
accordingly, we say that $C$ is \emph{Kunz} if all of its non-Gorenstein points are Kunz; we call $P$ \emph{almost Gorenstein} if $\mu_P=1$ and,
accordingly, we say that $C$ is \emph{almost Gorenstein} if all of its points are so. Following \cite[Dfn. 5.7]{KM}, we call $C$ \emph{nearly Gorenstein} if  $\mu=1$, i.e., $C$ is almost Gorenstein with just one non-Gorenstein point. Finally, following \cite[Dfn. 2.15]{KM}, we call $C$ \emph{nearly normal} if $h^0(\oo/\mathcal{C})=1$.}
\end{defi}

\begin{rem}
\label{remrel}
\emph{The relevance of the concepts above are summarized in three properties: (i) $C$ is nearly Gorenstein if and only if it is non-Gorenstein and $C'$ is projectively normal, owing to \cite[Thm. 6.5]{KM}; (ii) $C$  is nearly normal if and only if $C'$ is arithmetically normal, due to \cite[Thm. 5.10]{KM}; (iii) $P$ is Gorenstein if and only if $\eta_P=\mu_P=0$, and $P$ is non-Gorenstein if and only if $\eta_P,\mu_P>0$ by \cite[p. 438 top]{BF}; besides, if $\eta_P=1$ then $\mu_P=1$, by \cite[Prp. 21]{BF}. In particular, a Kunz curve with only one non-Gorenstein point is as close to being Gorenstein as it gets.}
\end{rem}

\subsection{Semigroup of Values} 

Now we establish few notations on evaluations. Given a unibranch point $P\in C$ and any function $x\in k(C)^*$, set
$$
v_{P}(x):=v_{\pb}(x)\in\zz
$$
where $\pb$ is the point of $\cb$ over $P$. The \emph{semigroup of values} of $P$ is
$$
\ssp:=v_{P}(\op ).
$$
We also feature two elements of $\sss$, namely:
\begin{equation}
\label{equaab}
\alpha_P :={\rm min}(\sss\setminus\{ 0\})\ \ \ \ \text{and}\ \ \ \beta_P :={\rm min}(v_P(\cp)).
\end{equation}
The set of \emph{gaps} of $\sss_P$ is
$$
{\rm G}_P:=\nn\setminus\sss_P
$$
and one defines the local invariant
$$
\delta_P:=\#({\rm G}_P)
$$
which agrees with the singularity degree of $P$, that is, $\delta_P=\dim(\obp/\op)$. The \emph{Frobenius vector} of $\sss_P$ is 
$$
\gamma_P :=\beta_P -1
$$ 
and one sets
\begin{equation}
\label{equkkp}
\kk_{P}:=\{ a\in\zz\ |\ \gamma_P -a\not\in\sss_P\}
\end{equation}
whose importance will appear later on. 


\subsection{Scrolls}
\label{secscr}

A \emph{rational normal scroll} $S:=S_{m_1,\ldots,m_d}\subset\pp^{N}$ with $m_1\leq\ldots\leq m_d$,  is a projective variety of dimension $d$ which, after a suitable choice of coordinates, is the set of points $(x_0:\ldots: x_N)\subset\mathbb{P}^N$ such that  the rank of
\begin{equation}
\label{equscr}
\bigg(
\begin{array}{cccc}
x_0 & x_1 & \ldots & x_{m_1-1} \\
x_1 & x_2 & \ldots & x_{m_1}
\end{array}
\begin{array}{c}
\big{|} \\
\big{|}
\end{array}
\begin{array}{ccc}
x_{m_1+1} & \ldots &  x_{m_1+m_2} \\
x_{m_1+2} & \ldots &  x_{m_1+m_2+1}
\end{array}
\begin{array}{c}
\big{|} \\
\big{|}
\end{array}
\begin{array}{c}
\ldots    \\
\ldots  
\end{array}
\begin{array}{c}
\big{|} \\
\big{|}
\end{array}
\begin{array}{cc}
\ldots & x_{N-1}  \\
\ldots & x_N
\end{array}
\bigg)
\end{equation}
is smaller than 2. So, in particular,
\begin{equation}
\label{equnnn}
N=e+d-1
\end{equation}
where $e:=m_1+\ldots+m_d$

Note that $S$ is the disjoint union of $(d-1)$-planes determined by a (parametrized) choice of a point in each of the $d$ rational normal curves of degree $m_d$ lying on complementary spaces on $\mathbb{P}^N$. We will refer to any of these $(d-1)$-planes as a \emph{fiber}. So $S$ is smooth if $m_i>0$ for all $i\in\{1,\ldots,d\}$. 
From this geometric description one may see that
\begin{equation}
\label{equdgs}
\deg(S)=e
\end{equation}
The scroll $S$ can also naturally be seen as the image of a projective bundle. In fact,  taking $\mathcal{E}:=\oo_{\pum}(m_1)\oplus\ldots\oplus\oo_{\pum}(m_d)$, one has a birational morphism  
\begin{equation*}
 \mathbb{P}(\mathcal{E})\longrightarrow S\subset\mathbb{P}^{N}
\end{equation*}
defined by $\oo_{\mathbb{P}(\mathcal{E})}(1)$.  The morphism is such that any fiber of $\mathbb{P}(\mathcal{E})\to\pum$ is sent to a fiber of $S$. It is an isomorphism if $S$ is smooth. One can check, for instance, \cite{EH,Rd} for more details. 

In this case, one may describe the Picard group of the scroll as 
$$
\text{Pic}(S)=\mathbb{Z}H\oplus\mathbb{Z}F
$$
where $F$ is the class of a fiber, and $H$ is the hyperplane class. One may also compute its Chow ring as
\begin{equation}
\label{equchw}
A(S)=\frac{\mathbb{Z}[H,F]}{(F^2\, ,\, H^{d+1}\, ,\, H^{d}F\, ,\, H^{d}-eH^{d-1}F)}
\end{equation}
From (\ref{equdgs}) we get the relations
\begin{equation}
\label{equrel}
H^d=e\ \ \ \ \ \text{and}\ \ \ \ \ H^{d-1}F=1
\end{equation}
The canonical class in $S$ is given by
\begin{equation}
\label{equccs}
K_S=-dH+(e-2)F
\end{equation}
By \cite[Lem. 3.1, Cor. 3.2]{Mr}, we also have the formulae
\begin{equation}
\label{equhhz}
h^0(\oo_S(aH+bF))=
\begin{cases}
\displaystyle (b+1)\binom{a+d-1}{d-1}+e\binom{a+d-1}{d}
&
{\rm if}\  a\geq 0\ \text{and}\ b\geq-am_1
\\
0 & \text{otherwise} 
\end{cases}
\end{equation}
and
\begin{equation}
\label{equhh1}
h^i(\oo_S(aH+bF))=0\ \ \ \ \ \text{if}\ i\geq 1,\  a\geq 0\ \text{and}\ b\geq -(am_1+1)
\end{equation}
which are important in the analisys of effective divisors on $S$.


\section{Canonical Models on Scrolls via Gonality}

In this section we analyze the relation between pencils on a curve and scrolls it might lie on. We are particularly interested when this curve happens to be a canonical model. So let $\dd(\aaa,V)$ be a pencil on a curve $C\subset\pp^n$; and assume $\aaa=\oo_C(D)$ where $D$ is an effective Weil divisor on $C$ supported outside ${\rm Sing}(C)$. Let also $H$ be a hyperplane divisor on $C$ and suppose the curve is linearly normal, that is, the hyperplane sections cut out a complete linear series. In this framework, one may adjust, for instance, Schreyer's survey in \cite[pp. 113-115]{Sc} to the singular case, in order to induce, by means of $\dd$, an inclusion $C\subset S\subset\pp^{n}$ where $S$ is a rational normal scroll.

In fact, consider then the multiplication map
\begin{equation}
\label{eqummp}
V\otimes H^0(\oo_C(H-D))\longrightarrow H^0(\oo_C(H))
\end{equation}
and assume $f:=h^0(\oo_C(H-D))\geq 2$. Then (\ref{eqummp}) yields a matrix in $M_{2\times f}(H^0(\oo_C(H)))$ whose $2\times 2$ minors vanish on $C$. Thus $C$ is contained in the rational normal scroll $S$ defined by these minors, which is such that
\begin{equation}
\label{equdsc}
\dim(S)=h^0(\oo_C(H))-h^0(\oo_C(H-D))
\end{equation}
We may apply this construction to the following example.

\begin{exem}
\label{exedif}
\emph{We already know from \cite{RS} that any trigonal canonical (Gorenstein) curve lies on a $2$-fold scroll.  So consider, for example, the curve 
$$
C=(1:t^3:t^6:t^7:t^9:t^{10})\subset\pp^5
$$
It is a rational monomial curve with just one singular point $P=(1:0:0:0:0:0)$. So its genus agrees with the singularity degree of $P$. Now $\sss_P=\langle 3,7\rangle$, hence $C$ has genus $g=\delta_P=6$. Moreover, $\deg(C)=10$ so $C$ is canonical. One can check that the curve is also trigonal, with the gonality computed by the linear series $|\oo_C(3Q)|$, where $Q=(0:0:0:0:0:1)$. So we may use the theory sketched above to place $C$ within a $2$-fold scroll. We have that $V=\langle 1,t^3\rangle$, $H=10\,Q$ and 
$$
H^0(\oo_C(H-3Q))=H^0(\oo_C(7Q))=\langle 1,t^3,t^6,t^7\rangle\subset k(t)=k(C)
$$
So, according to (\ref{equdsc}), $C$ lies on a $2$-fold scroll $S$ since 
$$
h^0(\oo_C(H))-h^0(\oo_C(7Q))=6-4=2
$$
In order to see this scroll in a way that it is defined by a matrix just like (\ref{equscr}), one may reorder the coordinates as
$$
C=(t^7:t^{10}:1:t^3:t^6:t^9)\subset\pp^5=\{(x_0:x_1:x_2:x_3:x_4:x_5)\}
$$
If so, $S$ is the surface of $\mathbb{P}^5$ cut out by the $2\times 2$ minors of
$$
\bigg(
\begin{array}{cccc}
x_0 & x_2 & x_3 & x_4 \\
x_1 & x_3 & x_4 & x_5
\end{array}
\bigg)
$$
Note that the scroll is of the form $S_{13}$.  Note also that the sections of the first row (when restricted to $C$) generate $H^0(\oo_C(7Q))= \langle 1\rangle \otimes H^0(\oo_C(7Q))$, while the sections of the second row generate $\langle t^3\rangle \otimes H^0(\oo_C(7Q))$ exactly as in the map (\ref{eqummp}).}

\emph{A small disturb on the above example is enough to realize how things get worse when one goes through the non-Gorenstein case. For instance, let now $C$ be a rational monomial curve with just one singularity whose semigroup is the same above up to the removal of $7$. Namely, a linearly normal model for such a curve with smaller possible dimension of the ambient space is
$$
C=(1:t^3:t^6:t^9:t^{10}:t^{12}:t^{13}:t^{14})\subset\pp^7
$$
The singular point is $P=(1:0:\ldots :0)$. Since $\sss_P=\langle 3,10,14\rangle$, the curve has genus $g=\delta_P=7$. And since $\beta_P=12\neq 2\delta_P$, it follows that $P$ is non-Gorenstein, and so is $C$. Moreover, the curve is also trigonal, with gonality computed by $|\oo_C(3Q)|$, where $Q=(0:\ldots:0:1)$, similar to the prior example. However, this pencil does not induce an inclusion of $C$ on a $2$-fold scroll, as would be expected. In fact, we have that $H=14Q$ and 
$$
H^0(\oo_C(H-3Q))=H^0(\oo_C(11Q))=\langle 1,t^3,t^6,t^9,t^{10}\rangle
$$
So, according to (\ref{equdsc}), $C$ lies on a scroll of dimension  
$$
h^0(\oo_C(H))-h^0(\oo_C(11Q))=8-5=3
$$
In order to achieve the expected dimension, one may deal with the canonical model $C'$ instead. Although $C'$ is not isomorphic to $C$ (since the latter is non-Gorenstein), it does preserve this desired property related to the gonality of $C$. Indeed, by Theorem \ref{prprat} ahead, we have that
$$
C'=(1:t^3:t^4:t^6:t^7:t^9:t^{10})\subset\pp^6
$$
which is clearly contained on a $2$-fold scroll of the form $S_{23}$. This inclusion of the canonical model $C'\subset S_{23}$ can be obtained, for instance, dealing with the pullback $|(\pi')^*(\oo_{C}(3Q))|$ and following the same steps above.}
\end{exem}

A general result relating pencils and scrolls for an arbitrary rational singular curve (not necessarily monomial)  with prescribed degree and ambient space dimension can be found, for instance, in \cite{CFM2} (motivated by \cite{CFM1}). However, if the concern is gonality (and hence canonical models), then the way of inducing an inclusion of $C'$ on a scroll by means of a $g_d^1$ on $C$ should be slightly modified. 
 A more intrinsic approach is required. This was done by St\"ohr and Rosa in \cite{RS} based on Andreotti and Mayer's \cite{AM} in the case $d=3$. The following result uses similar arguments to extend \cite[Thm. 2.1, Lem 2.3]{RS} to higher degree.  

\begin{teo}
\label{thmscr}
Let $C$ be an integral and projective curve of arithmetic genus $g$ over an algebraically closed field $k$, and $C'$ be its canonical model. Let $\dd(\aaa,V)$ be a $g_d^1$ on $C$. Then $\dd$ induces an inclusion $C'\subset S\subset\pp^{g-1}$ where $S$ is an $m$-dimensional rational normal scroll, such that:
\begin{itemize}
\item[(I)] $m\leq d-1$, with equality if and only if $\dd$ is complete;
\item[(II)] If $S$ is singular then
$$
\dim({\rm Sing}(S))< d-2h^0(\aaa)+1+\frac{\deg(\aaa\cap x^{-1}\aaa)}{2}
$$
with $x\in V\setminus k$; in  particular, 
\begin{itemize}
\item[(i)] if the expression above is not extrictly positive, then $S$ is smooth;
\item[(ii)] if $\dd$ is complete and base point free, then $\dim({\rm Sing}(S))\leq m-3$;
\end{itemize}
\item[(III)] If $C$ is Gorenstein and $\dd$ is complete with a base point then $S$ is singular.
\end{itemize}
\end{teo}

\begin{proof}
Without loss in generality one may write $V=\langle 1,x\rangle\subset H^0(\aaa)\subset k(C)$. Consider then the map
\begin{gather*}
\begin{matrix}
\varphi: & H^1(\aaa) & \longrightarrow & H^1(\oo_C) \\
    & f           & \longmapsto     & xf
\end{matrix}
\end{gather*}
defined for any $f\in{\rm Hom}(\aaa,\ww)$. Note that $\varphi$ is non-stable, i.e., for any subspace $W\in H^1(\aaa)$, if $\varphi(W)\subset W$, then $W=0$, because $\{x^if\}_{i\in\nn}$ form a linear independent set (viewd in $k(C)$). As $\oo_C$ is a subsheaf of $\aaa$, we have that $H^1(\aaa)\subset H^1(\oo_C)$. So, by \cite[Lem. 5, Cor. 1]{AM}, one may write
$$
H^1(\aaa)=\bigoplus_{i=1}^{r}\bigg(\bigoplus_{j=0}^{m_i-1}kx^jf_i \bigg)
$$
and
\begin{equation}
\label{equwws}
H^0(\ww)\cong H^1(\oo_C)=\bigg(\bigoplus_{i=1}^{r}\bigg(\bigoplus_{j=0}^{m_i}kx^jf_i \bigg)\bigg)\oplus\bigg(\bigoplus_{i=1}^{s}kh_i \bigg)
\end{equation}
where $f_i\in H^1(\aaa)\setminus\varphi( H^1(\aaa))$ for $1\leq i\leq r$ and $h_i\in H^1(\oo_C)\setminus(H^1(\aaa)+\varphi(H^1(\aaa)))$ for $1\leq i\leq s$. So by (\ref{equwws}) the canonical model $C'$ is the image of the morphism
$$
(f_1:\ldots:x^{m_1}f_1:\ldots\ldots:f_r:\ldots:x^{m_r}f_r:h_1:\ldots:h_s):\cb\longrightarrow\pp^{g-1}
$$
and, in particular,
$$
C'\subset S:=S_{m_1,\ldots,m_r,0,\ldots,0}\subset \pp^{g-1}
$$
Besides, 
\begin{align*}
\dim(S) &=r+s =\dim(H^1(\oo_C)/\varphi(H^1(\aaa)))\\
            &=h^1(\oo_C)-h^1(\aaa)\\
           &=g-(h^0(\aaa)-\deg(\aaa)-1+g)\\
           &=\deg(\aaa)-(h^0(\aaa)-1)
\end{align*}
where the third equality holds since $\varphi$ is injective. Now $\deg(\aaa)=d$ and 
$$
\dim(\dd)=2\leq h^0(\aaa)
$$
with equality holding if and only if $\dd$ is complete. So (I) is proved.

To prove (II), set $\G:=x^{-1}\aaa$, and note that $\varphi(H^1(\aaa))=xH^1(\aaa)=H^1(\G)$.  Now, for any subsheaves $\aaa$ and $\G$  of the constant sheaf $\mathcal{K}$ of rational functions of the curve $C$, one may form the short exact sequence
\begin{equation}
\label{equext}
0\longrightarrow \aaa\cap\G\longrightarrow \aaa\oplus\G\longrightarrow \aaa+\G\longrightarrow 0
\end{equation}
and apply the left exact functor $H^0(\mathcal{H}{\rm om}(\bullet,\ww))$ to get
$$
H^0(\mathcal{H}{\rm om}(\aaa,\ww))\cap H^0(\mathcal{H}{\rm om}(\G,\ww))=H^0(\mathcal{H}{\rm om}(\aaa+\G,\ww))
$$
from which we conclude that 
\begin{equation}
\label{equhh1}
H^1(\aaa)\cap H^1(\G)=H^1(\aaa+\G)
\end{equation}
From (\ref{equext}) we also have that
$
\chi(\aaa)+\chi(\G)=\chi(\aaa+\G)+\chi(\aaa\cap\G)
$
which yields
\begin{equation}
\label{equdeg}
\deg(\aaa+\G)=\deg(\aaa)+\deg(\G)-\deg(\aaa\cap\G)
\end{equation}
For $\aaa$ as in the statement of the theorem, and $\G$ as priorly set, it is easily seen by the definition of the latter that 
\begin{equation}
\label{equhgf}
h^0(\G)=h^0(\aaa)\ \ \ \ \ \ \text{and}\ \ \  \ \ \deg(\G)=\deg(\aaa)=d
\end{equation}
With this in mind, we have
\begin{align*}
\dim({\rm Sing}(S))&=s-1=\dim(H^1(\oo_C)/(H^1(\aaa)+\varphi(H^1(\aaa)))-1\\
                            &=\dim(H^1(\oo_C)/(H^1(\aaa)+H^1(\G)))-1\\
                           &=h^1(\oo_C)-(h^1(\aaa)+h^1(\G)-h^1(\aaa+\G))-1\\
                           &=g-(2(h^0(\aaa)-d-1+g)-h^1(\aaa+\G))-1\\
                           &=2(d-h^0(\aaa))-g+1+h^1(\aaa+\G)\\
                          &< 2(d-h^0(\aaa))-g+1+g-\frac{\deg(\aaa+\G)}{2}\\
                          &= 2(d-h^0(\aaa))+1-\bigg(\frac{2d-\deg(\aaa\cap\G)}{2}\bigg)\\
                         &=d-2h^0(\aaa)+1+\frac{\deg(\aaa\cap \G)}{2}
\end{align*}
where the fourth equality holds from (\ref{equhh1}), the fifth is due to (\ref{equhgf}) and Riemann-Roch, the unequality follows from \cite[App.]{EHKS}, and the seventh equality owes to (\ref{equdeg}). So item (II).(i) follows. Now, if $\dd$ is base point free, then $\aaa\cap\G=\oo$, and if it is complete, then $h^0(\aaa)=2$, so (II).(ii) follows as well.

To prove (III), the natural isomorphism
$
\mathcal{H}{\rm om}(\aaa\cap\G,\ww)=\mathcal{H}{\rm om}(\aaa,\ww)+\mathcal{H}{\rm om}(\G,\ww)
$
yields the inclusion $H^1(\aaa)+H^1(\G)\subset H^1(\aaa\cap\G)$, thus
\begin{align*}
\dim(H^1(\oo_C)/(H^1(\aaa)+H^1(\G))) &\geq \dim(H^1(\oo_C)/H^1(\aaa\cap\G))\\
                                                                   &=g-(h^0(\aaa\cap\G)-\deg(\aaa\cap\G)-1+g)\\
                                                                   &=\deg(\aaa\cap\G)+1-h^0(\aaa\cap\G)
\end{align*}
If $\dd$ is complete, then $h^0(\aaa)=2$; but since $H^0(\aaa\cap\G)\subset H^0(\aaa)$ and $x\not\in H^0(\aaa\cap\G)$ it follows that $h^0(\aaa\cap G)=1$. On the other hand, if $\dd$ is has a base point $P$, which is Gorenstein since $C$ is so, then $\aaa_P\cap\G_P\supsetneq \oo_P$, thus $\deg(\aaa\cap\G)>0$ and hence $s>0$, i.e., $S$ is singular.
\end{proof}

As a consequence of the above result we have the following.

\begin{coro}
\label{corida}
Let $C$ be an integral and projective curve of gonality $d$ and $C'$ be its canonical model. Then $C'$ lies on a $(d-1)$-fold scroll.
\end{coro}

\begin{proof}
By item (I) of the prior theorem, it suffices to prove that if a linear series $\dd$ computes the gonality of $C$, then it is complete. So write $\dd=(\aaa,V)$ where $\aaa$ is torsion free of rank $1$. Choose any regular point $P\in C$ and consider the sequence
$$
0\longrightarrow \aaa(-P)\longrightarrow \aaa \longrightarrow \aaa/\aaa(-P)\longrightarrow 0
$$
Taking Euler characteristic yields 
$$
\big(h^0(\aaa)-h^0(\aaa(-P))\big)+\big(h^1(\aaa(-P))-h^1(\aaa)\big)=1
$$
Now both summands are nonnegative; besides, $h^0(\aaa)\geq \dim(V)=2$ and we also have that $h^0(\aaa(-P))\leq 1$ because this sheaf is of degree $d-1$ and the gonality of $C$ is $d$. Thus $h^0(\aaa)=2$ as desired, i.e., $\dd$ is complete.
\end{proof}

\section{Gonality via Canonical Models on Scrolls}

The point of departure of this section is the assumption that the canonical model $C'$ is a complete intersection inside a smooth scroll $S$. From this, we derive some properties about $C$ in terms of invariants of $S$ and $C'$.  

To begin with, let $X$ be a curve lying on a $d$-dimensional smooth variety $S$ as
$$
X=D_1\cdot\ldots\cdot D_{d-1}
$$
where the $D_i$'s are divisors on $S$, and set $\mathcal{E}:=\oplus_{i=1}^{d-1}\oo_S(D_i)$.
In order to compute the arithmetic genus of $X$, consider its resolution inside $S$ via Koszul Complex given by
$$
0\rightarrow\bigwedge^{d-1}\mathcal{E}^{\vee}\to\ldots\to
\bigwedge^2\mathcal{E}^{\vee}\to\mathcal{E}^{\vee}\to\oo_S\to\oo_X\to 0
$$
It yields
\begin{equation}
\label{equerr}
p_a(X)=1-\chi(\oo_S)+\chi(\mathcal{E}^{\vee})-\chi(\wedge^2\mathcal{E}^{\vee})+\ldots+ (-1)^{d}\chi(\wedge^{d-1}\mathcal{E}^{\vee})
\end{equation}
Now
\begin{equation}
\label{equext}
\bigwedge^j\mathcal{E}^{\vee}=
\bigoplus_{1\leq i_1<\ldots <i_j \leq d-1}\oo_S(-D_{i_1}-\ldots -D_{i_j}) 
\end{equation}
and for an arbitrary divisor $D\in{\Pic}(S)$, Hirzebruch-Riemann-Roch Theorem yields
\begin{equation}
\label{equeul}
\chi(\oo_S(D))=t_d+ D\cdot t_{d-1} + \ldots +\frac{1}{(d-2)!}D^{d-2}\cdot t_2+\frac{1}{(d-1)!}D^{d-1}\cdot t_1+\frac{1}{d!}D^{d}
\end{equation}
where $t_i$ is the $i$th degree component of the Todd class of the tangent bundle $\mathcal{T}_S$.

On the other hand, note that whatever are  $\alpha_1,\ldots ,\alpha_s\in A$, a commutative ring, we have that
\begin{equation}
\label{equcrg}
\sum_{j=1}^{s}\bigg((-1)^{j-1}\sum_{1\leq i_1<\ldots <i_j \leq s}(\alpha_{i_1}+\ldots+\alpha_{i_j})^r\bigg) = 
\begin{cases}
0    & r<s\\
(-1)^{s+1}s!\,\alpha_1...\,\alpha_s& r=s \\
\displaystyle\frac{(-1)^r r!}{2}\sum_{i=1}^{s}\alpha_1...\alpha_i^2...\alpha_s &r=s+1 
\end{cases} 
\end{equation}
So apply (\ref{equeul}) to (\ref{equext}) using the linearity of the Euler characteristic. Then write the sum envolving the exterior powers of $\mathcal{E}^{\vee}$ in (\ref{equerr}) as a linear polynomial on the variables $t_i$ and apply (\ref{equcrg}) to each of its coefficients. If so, we are reduced to
$$
p_a(X) = 1 - \chi(\oo_S)+\bigg (t_d-D_1 \cdot\ldots\cdot D_{d-1}\,t_{1}+\frac{\sum_{i=1}^{d-1}D_1\cdot\ldots\cdot D_{i}^2\cdot\ldots D_{d-1}}{2}\bigg) 
$$
Now recall that $t_d=\chi(\oo_S)$ and $t_1=c_1(\mathcal{T}_S)/2$, which yields the formula
\begin{equation}
\label{equger}
2p_a(X)-2= D_1 \cdot\ldots\cdot D_{d-1}\cdot\big(D_1+\ldots+D_{d-1}-c_1(\mathcal{T}_S)\big)
\end{equation}

Now assume $S$ is a rational normal scroll and recall the settings of Section \ref{secscr}. To any curve $X\subset S$ consider the parameter
\begin{equation}
\label{equlll}
\ell:=X\cdot F
\end{equation}
which will be widely studied here. If $X$ is a complete intersection inside $S$ write
$$
D_i=a_iH+b_iF
$$
as above. Then, from (\ref{equchw}) and (\ref{equrel}), we easily get
\begin{equation}
\label{equcll}
\ell=a_1\cdot\ldots\cdot a_{d-1}
\end{equation}
and one may also use these relations to compute $X\cdot H$ and obtain
\begin{equation}
\label{equcdd}
\deg(X)=a_1\cdot\ldots\cdot a_{d-1}\cdot e+\sum_{i=1}^{d-1}a_1\cdot\ldots\cdot b_{i}\cdot\ldots\cdot a_{d-1}
\end{equation}
For the remainder and for the sake of simplicity, write
$$
a:= a_1+\ldots+a_{d-1}\ \ \ \ \ \ \ \ \ \ \ \ b:= b_1+\ldots+b_{d-1}
$$
If so, we will refer to $X$ as being of \emph{$(a,b)$-type}, and note that
$$
D_1+\ldots+D_{d-1}=aH+bF
$$
while from (\ref{equccs}) we have that
$$
c_1(\mathcal{T}_S)=dH+(2-e)F
$$
So to compute the arithmetic genus of $X$, we get
\begin{align*}
D_1 \cdot\ldots\cdot D_{d-1}\cdot c_1(\mathcal{T}_S) & =(de+2-e)\,a_1\cdot\ldots\cdot a_{d-1}+d\,\sum_{i=1}^{d-1}{a_1\cdot\ldots\cdot b_i \cdot\ldots\cdot a_{d-1}}\\
                                                                          & = d\deg(X)+(2-e)\ell
\end{align*}
and similarly
$$
D_1 \cdot\ldots\cdot D_{d-1}\cdot (D_1+\ldots+D_{d-1})=a\deg(X)+ b\,\ell
$$
which, combined with (\ref{equger}), yield
\begin{equation}
\label{equgsc}
2p_a(X)-2= \deg(X)(a-d)+\ell(b+e-2)
\end{equation}
that is a helpful tool to be used here applied to canonical models of curves.

In the following result we generalize, to arbitrary dimension, \cite[Thms. 2.1, 4.1]{LMM} which study curves $C$ for which the canonical model $C'$ lies on a $d$-fold scroll for $d=2,3$. The idea is to put the statement within a way that both cases can be deduced from general formulae involving $d$; and the reason why focusing just on the cases $\ell=1,2,d,d+1$ in item (II) will be clear after Theorem \ref{thmth3}.

\begin{teo} 
\label{thmth2}
Let $C$ be a nonhyperelliptic curve of genus $g\geq d+3$, whose canonical model $C'$, of genus $g'$, lies on a smooth $d$-dimensional rational normal scroll $S$ as a complete intersection of $(a,b)$-type. Let $\ell$ be the number of points of $C'$ in a generic fiber of the scroll. Then the following hold:

\begin{itemize}
\item[(I)] $\gon(C)\leq \ell+{g-g'}$
\item[(II)] If $b=-(g-(d+2))$ then either $C$ is Gorenstein at most tetragonal or, else, the equality holds if and only if $\ell=2$ and $C'$ is elliptic. Otherwise
$$
\ell= \frac{(a-d-1)(2g-2-\eta)+\eta +2\mu}{b-d-2+g}
$$
in particular, $\ell\leq 3$ if $d=2$. 
\item[(III)] The following hold:
\begin{itemize}
\item[(i)] if $\ell=1$ then $C'\cong\proje^{1}$ and $C$ is rational with all singular points non-Gorenstein; 

\item[(ii)] if $\ell=2$ and  $C'$ is not elliptic, then $C'\cong\proje^1$ iff $b=-(g-(d+1))$ or else $b\geq -(g-(d+3))$.

\item[(iii)] if $\ell=d$, then $b=-(g-(d+2))-(\eta+2\mu +\tau(2g-2-\eta))/d$; where $\tau\in\{-1,0,1,\ldots,d-3\}$. In particular, if $d=2$ then $C$ is nearly Gorenstein; and if $d=3$, then $C$ is Kunz with just one non-Gorenstein point iff $b=-(g-4)$.
\item[(iv)] if $\ell=d+1$, then $b=-(g-(d+2))-(\eta+2\mu +\tau(2g-2-\eta))/(d+1)$; where $\tau\in\{0,1,\ldots,d-2\}$. In particular, if $d=2$ then $C$ is almost Gorenstein if and only if it is Kunz; and if $d=3$, then $C$ is Kunz with just one non-Gorenstein point iff $b=-(3g-2)/2$ with $g$ even.
\end{itemize}
\item[(IV)] If $C$ is non-Gorenstein and $\ell\geq 3$, writing $S=S_{m_1,\ldots,m_d}$ we have
$$
\frac{g-d-1}{\ell+d-2}+\frac{\nu (g-1)+3-\ell}{\ell(d+\ell-2)}\leq m_1\leq\ldots\leq m_{d}\leq\frac{2g-2-\eta}{\ell}  
$$
where $\nu:=(d-1)\sqrt[d-1]{\ell}-d-1$; unless $a=d+1$ where $m_1\geq (g-d-1)/(d+1)$;  in particular, both formulae extend the case $d=2$ where $m_1\geq (g-3)/3$.  
\end{itemize}
\end{teo}

\begin{proof}
We start by pointing out that item (I) and  and the upper bound in (IV) do not actually depend on $C'$ being a complete intersection in $S$. So we begin our proof by these parts.  To see (I), first note that, since $S$ is nonsingular, $\ell$ agrees with the intersection number $C'\cdot F$ priorly defined. Moreover, the fibers of $S$ cut out a $g_{\ell}^1$ which we may write as $\dd(\oo_{C'}(D),V)$. We may assume $D$ is effective, and considering the following sequence
$$
0\longrightarrow \oo_C\longrightarrow \pi'_*(\oo_{C'}(D)) \longrightarrow \pi'_*(\oo_{C'}(D))/\oo_C\longrightarrow 0
$$
from which we get, after taking Euler characteristic, that
$$
\deg(\pi'_*(\oo_{C'}(D)))=h^0(\pi'_*(\oo_{C'}(D))/\oo_C)
$$
We may further assume that $D$ is supported outside $(\pi')^{-1}({\rm Sing}(C))$ to see that
$$
h^0(\pi'_*(\oo_{C'}(D))/\oo_C)=\deg(D)+h^0(\pi'_*(\oo_{C'})/\oo_C)=\ell+g-g'
$$
which is the degree of the pencil $\dd(\pi'_*(\oo_{C'}(D)),V)$ on $C$ and, in particular, beats its gonality. To check the upper bound in (IV), follow \cite[pp. 113-115]{Sc} to see that $\oo_{C'}(m_dD)\subset \oo_{C'}(H)$. Now  $\deg(\oo_{C'}(m_dD))=m_d\ell$; while, on the other hand, $\oo_{C'}(H)=\oo_{C'}\ww=\oo_{\widehat{C}}\ww$ since $C'\cong\widehat{C}$. But $\deg(\oo_{\widehat{C}}\ww)=2g-2-\eta$ as proved in \cite[Cor. 4.9]{KM} and the bound follows comparing degrees.

To prove the remaining items, we will take $X = C'$ in (\ref{equgsc}). If so, note that $p_a(C')=g-\eta-\mu$ due to (\ref{equget}), and $\deg(C')=\deg(\oo_{C'}(H))=2g-2-\eta$ as we just have seen. Besides, the ambient dimension is $N=g-1$ which implies that $e=g-d$ by (\ref{equnnn}). Thus (\ref{equgsc}) reduces to
\begin{equation}
\label{equpacan}
(a - d -1)(2g - 2 - \eta) + \eta + 2\mu - \ell(d+2-b-g) = 0
\end{equation}

If $b \neq d+2-g$, the above equality provides the formula of (II). Otherwise
\begin{equation}
\label{equgn1}
\eta + 2\mu = (d+1 -a)(2g - 2 - \eta)
\end{equation}
Since the left hand side of the equation is positive, we have that $a\leq d+1$. Now note that $C'$ is given by the intersection of effective divisors, so all $a_i\geq 0$ due to (\ref{equhhz}). Moreover, by its very construction, $C'$ is nondegenerate; therefore $\ell\geq 1$, which implies that all $a_i\geq 1$ and hence $a\geq d-1$. Being nondegenerate in $\mathbb{P}^{g-1}$, it follows that $C'$ has degree at least $g-1$. So if all $a_i=1$, then (\ref{equcdd}) yields
\begin{align*}
\deg(C')&=e+b=(g-d)-(g-d-2)=2
\end{align*}
which is precluded since $g\geq d+3\geq 5$. Therefore at least one $a_i\neq 1$ and hence $a\geq d$. Now $a=d$ if and only if there is exactly one $a_i\neq 1$, and it must be $2$, which holds if and only if $\ell=2$; besides (\ref{equgn1}) establishes the relation
$$
g-\eta-\mu=1
$$
which, by (\ref{equget}), is equivalent to $g'=1$, i.e., $C'$ is elliptic. If $a=d+1$, then either there are just two $a_i\neq 1$ both with value $2$, so $\ell=4$, or there is exactly one $a_i\neq 1$ and with value $3$, so $\ell=3$; besides the vanishing of the left hand side of (\ref{equgn1}) implies that $C$ is Gorenstein. In this case, $C\cong C'$ and, recalling that the fibers cut out a $g_{\ell}^1$ in $C'$, the gonality of $C$ is at most $\ell$, which, as just seen, should be $3$ or $4$ in such a case. So the first statement of (II) is proved up to sufficiency, which will be analyzed right away when we deal with the case $\ell=2$.

To carry out (II), it remains proving that $\ell\leq 3$ if $d=2$. In fact, $d=2$ is a special case where $a=\ell$ and $b=d'-\ell(g-2)$ with, say, $d':=\deg(C')$, owing to (\ref{equcdd}). So (\ref{equgsc}) turns into
\begin{equation}
\label{equll2}
2g'=p(\ell)=-(g-2)\ell^2+(2d'+g-4)\ell-2(d'-1)
\end{equation}
Since $y=p(x)$ is concave downwards with roots $1$ and $2(d'-1)/(g-2)$, and $g'$ is positive, it follows that
$$
\ell\leq \frac{2(d'-1)}{g-2}=\frac{2(2g-\eta-3)}{g-2}=4-\frac{2(\eta-1)}{g-2}
$$
So $\ell\leq 4$ since $g\geq 5$. If $C$ is Gorenstein then $g'=g$, so taking $\ell=4$ in (\ref{equll2}) yields $g=3$ which is precluded. And if $C$ is not Gorenstein and $\ell=4$ then $\eta=1$ and thus $g'=0$; but by Remark \ref{remrel}.(iii), $\mu=1$ as well so $g=2$ which cannot happen either. It follows that $\ell\leq 3$ in any case.

To prove (III), assume $\ell=1$. If so, the fibers of $S$ cut out a $g_1^1$ on $C'$, so $C' \cong \proje^1$. Since $C$ is nonhyperelliptic, $C$ and $C'$ are birationally equivalent, so $C$ is rational. Moreover, as $C'$ is nonsingular, then all singular points of $C$ must be non-Gorenstein, because $\oo_{C',P}\cong\oo_{C,P}$ if $P$ is Gorenstein and $C$ nonhyperelliptic. 

If $\ell=2$ then, as said above, $a=d$ and (\ref{equpacan}) yields
$$
\eta + \mu + b - d -1 = 0
$$
So one may write $b = d-g+g'+1$. If $g'=0$ then $C' \cong \proje^1$ with $b =d+1-g$; if $g'= 1$, that is, $C'$ is elliptic, then $b =d+2-g$ and the equivalence in (II) is now accomplished. And, finally, if $g'\geq 2$, it follows that $b>d+2-g$ as desired. 

If $\ell=d$ then (the formula in) (iii) follows from (\ref{equpacan}) setting $\tau:=a-d-1$ and observing that $a\geq d$, as seen above, and $a\leq \ell+d-2=2d-2$. In particular, if $d=2$ then $\tau=-1$, and, as said above, $b=d'-\ell(g-2)=2-\eta$, so replacing this in (iii) yields $\mu=1$, which is equivalent to saying that $C$ is nearly Gorenstein. And if $d=3$ then either $\tau=-1$, which implies that $a=d$ from which we get $\ell=2$ as priorly discussed; but this is precluded since $\ell=d=3$; hence $\tau=0$ which yields $b=-(g-4)+(\eta+2\mu)/3$; the result follows since $C$ is Kunz with just one non-Gorenstein point iff $\eta=\mu=1$.

If $\ell=d+1$ then, similarly, (iv) follows from (\ref{equpacan}) setting agian $\tau:=a-d-1$ and observing that $a\geq d+1$ and $a\leq \ell+d-2=2d-1$. In particular, if $d=2$ then $\tau=0$, and now $b=d'-\ell(g-2)=4-g-\eta$, so replacing this in (iv) yields $\eta=\mu$, which implies that $C$ is almost Gorenstein iff it is Kunz. And if $d=3$ then either $\tau=0$, which implies that $b=-g+5-(\eta+2\mu)/4$ which precludes the case $\eta=\mu=1$ since $b$ is an integer, or, else $\tau=1$ and $b=-(3g-2)/2$ with $g$ even for $C$ to be Kunz with just one non-Gorenstein point.

To finish the proof, recall that $m_1\geq -b/a$ by (\ref{equhhz}). Since we are assuming $\ell\geq 3$, we may take $a\geq d+1$ as discussed above. If equality holds then (\ref{equpacan}) yields
$$
m_1\geq \frac{g - d - 1}{d+1} + \frac{\eta + 2\mu}{\ell}-\frac{1}{d+1}
$$
and the lower bound follows disregarding the sum of the last two terms, which is positive since $C$ is non-Gorenstein. On the other hand, if $a> d+1$ then (\ref{equpacan}) yields
$$
m_1 \geq \frac{g - d - 2}{a}+\frac{(a - d - 1)d' + \eta + 2\mu}{a\ell}
$$
and the bound follows since $d'\geq g-1$ as seen above, $\eta+2\mu\geq 3$ because $C$ is non-Gorenstein, and $(d-1)\sqrt[d-1]{\ell} \leq a \leq \ell+d-2$ by definition.
\end{proof}

As pointed out in \cite[Thm 4.1]{LMM}, if $d=3$ a similar result can be obtained relaxing the hypothesys to $C'$ being just a local complete intersection. This was done via Hartshorne-Serre correspondence for varieties of codimension $2$ lying on a smooth ambient. So we may complete the tableaux in \cite[Sec. 5]{LMM} of all tetragonal rational monomial curves by exhibiting chart equations adjusting the methods of \cite{Br}. The parameter $\ell$ defined above can be computed independently for such a curve. In fact, one may write an affine chart of the scroll $S_{m_1\ldots m_{d}}$ where $C'$ lies as
$$
(1:x:\ldots :x^{m_{d}}:y_1:y_1x:\ldots:y_1x^{m_{d-1}}:\ldots :y_{d-1}:y_{d-1}x:\ldots: y_{d-1}x^{m_1})\cong\mathbb{A}^{d}
$$
This yields morphisms
$$
\varphi_i: C'\longrightarrow \mathbb{P}^1\cong (1:x:\ldots:x^{m_1})
$$
and hence $\ell$ agrees with the generic number of points in a fiber of $\varphi_i$ no matter is $i\in\{1,\ldots,d\}$. If $C$ is rational monomial, then one may write $x=t^r$ and $y_i=t^{s_i}$. So the number of points in the pre-image of $(1:a:\ldots:a^{m_i})$ is precisely the number of solutions of the equation $t^r=a$ because any $t$ determines a unique point of $C'$ since it is parametrized. Thus $\ell=r$ and, by construction, this $r$ agrees with the same common difference of the arithmetic progression of Lemma \ref{Danielle}. The following tableau improves then \cite[pp. 18-19]{LMM}, where, for any rational monomial tetragonal curve $C$ with just one singular point of genus at most $8$, is given the canonical model $C'$ and a scroll $S$ it lies on. Here we also give the equations of $C'$ in $S$, the parameter $\ell$, and the kind of singularities.

\begin{center}
\begin{tabular}{|c|c|c|c|c|}
\hline
\multicolumn{5}{|c|}{\textbf{genus 6}}\\
\hline
$C$ and $C'$ & & eqs for $C'$ &  $\ell$ & scr \\
\hline
 $(1:t^5:t^6:t^8:t^{13}:t^{14})$ & -- & $f=x^3-z$                    & $2$    & $S_{111}$ \\   
$(1:t^2:t^5:t^6:t^7:t^8)$	          &    & $h=y^6-3x^2y^4z$ &        &        \\
	          &    & \ \ \ \ \ \ \ \ \ \ \ \ \ $+3x^4y^2z^2-z^5$ &        &         \\ \hline  
\multicolumn{5}{|c|}{\textbf{genus 7}}\\
\hline
$(1:t^4:t^7:t^{12}:t^{13})$           &  --	 & $f=x^4-y$                  & $1$    & $S_{112}$ \\
$(1:t:t^4:t^5:t^7:t^8:t^9)$	 &        & $h=x^7-z$ &        &            \\
\hline 
     $(1:t^4:t^{10}:t^{11}:t^{12}:t^{13})$ & --  & $f=y^2-x$                   & $4$    & $S_{112}$ \\   
$(1:t^2:t^3:t^4:t^6:t^7:t^8)$	             &   & $h=z^4-2xyz^2+x^3$           &        &           \\ \hline
$(1:t^5:t^8:t^{11}:t^{12}:t^{13}:t^{14})$ & -- & $f=x^3-z$	   &	$2$		 & $S_{112}$ \\
$(1:t^2:t^3:t^5:t^6:t^7:t^8)$		           &   & $h=y^2-z$                 &        &           \\ 
\hline
$(1:t^5:t^7:t^8)$  & --  & $f=x^4-z$ &	$2$		 & $S_{112}$ \\
$(1:t^2:t^5:t^7:t^8:t^9:t^{10})$ &    & $h=y^8-4xy^6z+6x^2y^4z^2$  &        &           \\
	&                                  &\ \ \ \ \ \ \ \ \ \ \ \ \ \ \ \ \  $-4x^3y^2z^3+z^5$     &        &           \\ 
\hline
 	 $(1:t^6:t^7:t^8:t^{10})$                 &  -- & $f=x^3-y$	&	$2$		 & $S_{112}$ \\
$(1:t^2:t^6:t^7:t^8:t^9:t^{10})$ &                                  & $h=y^7-3x^2y^5z^2$       &        &           \\
 &                                  & \ \ \ \ \ \ \ \ \ \ $+3x^4y^2z^4+z^6$       &        &           \\
\hline
\multicolumn{5}{|c|}{\textbf{genus 8}}\\
\hline
 $(1:t^4:t^{10}:t^{13}:t^{14}:t^{15})$    & -- & $f=x^3-z$ 					    &	$2$		 & $S_{222}$ \\
$(1:t^2:t^4:t^5:t^6:t^7:t^8:t^9:t^{10})$					                                  &                                          & $h=y^6-3x^2y^4z$       &        &           \\
			&                                          &\ \ \ \ \ \ \ \ \ \ \ \ \ \ \ \ $+3x^4y^2z^2+z^5$ &        &   \\
\hline        
$(1:t^6:t^9:t^{11}:t^{13}:t^{14}:t^{15}:t^{16})$ & -- & $f=x^3-z$ 			 &	$2$		 & $S_{113}$ \\
$(1:t^2:t^3:t^5:t^6:t^7:t^8:t^9)$ &                           & $h=y^2-z$              &      &     \\ 
\hline
 $(1:t^4:t^9:t^{14}:t^{15})$     &   --          & $f=x^4-y$ 					    &	$1$		 & $S_{122}$ \\
  $(1:t:t^4:t^5:t^6:t^8:t^9:t^{10})$   &                                          & $h=y^2-z$              &        &      \\
\hline
$(1:t^4:t^9:t^{14}:t^{15})$              &  --      & $f=x^4-y$ 					    &	$4$		 & $S_{122}$ \\
$(1:t:t^4:t^5:t^6:t^8:t^9:t^{10})$					                                  &                                          & $h=x^6-z$     &        &           \\ 
\hline
$(1:t^5:t^7:t^{13}:t^{15}:t^{16})$       & -- & $f=x^4-z$ 					    &	$2$		 & $S_{113}$ \\
$(1:t^2:t^3:t^5:t^7:t^8:t^9:t^{10})$                &                                          & $h=y^2-x^3$            &        &           \\ 
\hline
$(1:t^5:t^7:t^9:t^{10})$                 & -- & $f=x^5-z$ 					  &	$2$		 & $S_{113}$ \\
$(1:t^2:t^5:t^7:t^9:t^{10}:t^{11}:t^{12})$     &                                          & $h=y^2-z$              &        &           \\ 
\hline
$(1:t^4:t^{10}:t^{11}:t^{16}:t^{17})$    & NG & $f=x^3-y^2$			    &	$4$		 & $S_{113}$ \\
$(1:t^4:t^6:t^7:t^8:t^{10}:t^{11}:t^{12})$         &          & $h=y^7-3xy^4z$\ \ \ \ \ \ \ \ \          &        &           \\
        &          & \ \ \ \ \ \ \ \ \ $+3x^2yz^4+z^6$         &        &           \\
\hline
$(1:t^4:t^9:t^{11}:t^{15}:t^{16})$       & K & $f=x^4-yz$  			    &	$4$		 & $S_{113}$ \\
$(1:t^4:t^7:t^8:t^9:t^{11}:t^{12}:t^{13})$       &        & $h=y^9-4x^3y^6z$\ \ \ \ \ \ \ \  \ \ \     &        &           \\ 
 	                                          &                                          &\ \ \ \ \ \ $+6x^6y^3z^2-4x^9z^3+z^7$&     &        \\  
\hline
 $(1:t^4:t^{11}:t^{13}:t^{14})$           &      --    & $f=x^3-y$  			      &	$1$		 & $S_{122}$ \\
$(1:t:t^3:t^4:t^5:t^7:t^8:t^9)$     &                                          & $h=x^7-z$      &        &           \\ 
\hline
$(1:t^4:t^{11}:t^{13}:t^{14})$           &     --     & $f=x^3-y$  			      &	$4$		 & $S_{122}$ \\
$(1:t:t^3:t^4:t^5:t^7:t^8:t^9)$      &                                          & $h=x^4-z$     &        &           \\ 
   \hline
$(1:t^5:t^8:t^{12}:t^{13}:t^{14})$       &  --   & $f=x^4-z$  			      &	$2$		 & $S_{122}$ \\
$(1:t^2:t^4:t^5:t^7:t^8:t^9:t^{10})$    &                                          & $h=y^8-4xy^6z+6x^2y^4z^2$&      &           \\ 
 	                                          &                                          &\ \ \ \ \ \ \ \ \ \ \  \ \  $-4x^3y^2z^3+z^5$     &        &           \\ 
\hline
\end{tabular}
\end{center}

\

Note that in the tableau of \cite[p. 11]{LMM} one finds no trigonal nearly Gorenstein (or Kunz) curve $C$ for which the canonical model $C'$ lies on a smooth surface scroll with $\ell=1$. And in the above tableau, one also finds any tetragonal nearly Gorenstein $C$ with $C'$ lying on a $3$-fold scroll with $\ell=1,2$. This motivates the following result.

\begin{teo} 
\label{thmth3}
Let $C$ be a nearly Gorenstein curve, whose canonical model $C'$ lies on a smooth $d$-dimensional rational normal scroll $S$. If the fibers of $S$ cut out a complete $g_{\ell}^1$ on $C'$, then $d\leq \ell\leq d+1$.
\end{teo}

\begin{proof}
Since $S$ is smooth, the $g_{\ell}^1$ is base point free and can be written as $\dd(\oo_{C'}(D),V)$. Now, according to \cite[Thm. 6.5]{KM}, $C$ is nearly Gorenstein if and only if $C'$ is linearly normal. And if $C'$ is linearly normal and lies on $S$, then,  by \cite[Sec. 2]{Sc}, we have that $h^0(\oo_{C'}(H-D))=\deg(S)$. Using \cite[Cor. 4.9]{KM} along with Riemann-Roch in $C'$ for the left hand side of the latter equality, while using (\ref{equnnn}) and (\ref{equdgs}) for the right hand side we get
\begin{equation}
\label{equff1}
2g-2-\eta-\ell+1-g'+h^1(\oo_{C'}(H-D))=g-d
\end{equation}
Recalling that $g'=g-\eta-\mu$ and that $C$ is nearly Gorenstein if and only if $\mu=1$, we have that (\ref{equff1}) turns into
\begin{equation}
\label{equff2}
\ell=d+h^1(\oo_{C'}(H-D))
\end{equation}
Set $\aaa:=\pi'_{*}(\oo_{C'}(D))$. Pushing forward and using Serre duality on $C$ we get
\begin{align*}
H^1(\oo_{C'}(H-D))&=H^1(\pi'_{*}(\oo_{C'}(H-D)))\\
                                &=H^1(\ooh\ww\otimes\aaa^{\vee})\\
                                &=H^0(\mathcal{H}{\rm om}(\ooh\ww,\ww)\otimes\aaa)
\end{align*}
Set $\G:=\mathcal{H}{\rm om}(\ooh\ww,\ww)\otimes\aaa$. Clearly, $H^0(\G)\subset H^0(\aaa)=H^0(\oo_{C'}(D))$. Now, since $C$ is nearly Gorenstein, it has a unique non-Gorenstein point, say $P$. We may assume that $D$ is supported outside $(\pi')^{-1}(P)$. So, by construction, $\G_Q=\aaa_Q$ if $Q\neq P$, while $\G_P={\rm Hom}_{\op}(\ooh_P,\op)$. We may further assume that $D$ is effective, so $k\subset H^0(\aaa)$; but note that $k\cap\G_P=0$, so $H^0(\G)\subsetneq H^0(\aaa)$. Therefore 
$$
h^1(\oo_{C'}(H-D))=h^0(\G)<h^0(\aaa)=h^0(\oo_{C'}(D))=2
$$
where the last equality holds since the $g_{\ell}^1$ is complete. The result follows by (\ref{equff2}).
\end{proof}

As observed in \cite{LMM}, the efforts we make here to charcterize gonality via scrolls, though general, are still incipient. A detailed study of syzygies in the same way of, e.g., \cite{Sc} or \cite{CS} (and references therein), will likely be required; or even adjusting the techniques of the recent \cite{HU} (and references therein) to canonical models, and also scrolls as the ambient space.

\section{The Gonality of Rational Monomial Curves}

Now we set the objects we'll be dealing with in this section. Consider the morphism

\begin{gather*}
\begin{matrix}
\mathbb{P}^1 & \longrightarrow & \mathbb{P}^n \\
   (s:t)              & \longmapsto     & (s^{a_n}:t^{a_1}s^{a_n-a_1}:\ldots:t^{a_{n-1}}s^{a_n-a_{n-1}}:t^{a_n})
\end{matrix}
\end{gather*}
The image $C$ of such a map is what we call a \emph{rational monomial curve}, which, for simplicity, we denote by
$$
C=(1:t^{a_1}:\ldots:t^{a_{n-1}}:t^{a_n})
$$
with $a_1<\ldots < a_n$. Note that $C$ admits at most two singular points which are
$$
P=(1:0:\ldots:0)
$$
the origin of the affine space $\mathbb{A}^n\subset \mathbb{P}^n$, and
$$
Q=(0:0:\ldots : 1)
$$
the point at the infinity under the parametrization given by $t$. The following result generalizes \cite[Prp. 3.1]{LMM} and by means of rather different arguments.

\begin{teo}
\label{prprat}
Let $C$ be a rational monomial curve. Then its canonical model is
$$
C'=(1:t^{b_2}:\ldots :t^{b_{\delta_P}}:t^{c_1}:\ldots:t^{c_{\delta_Q}})
$$
where $\{0,b_2,\ldots,b_{\delta_P}\}=\gamma_P-G_P$ and $\{c_1,\ldots,c_{\delta_Q}\}=\gamma_P+G_Q$.
\end{teo}

\begin{proof}
Let $g$ be the genus of $C$ and let $\vv$ be a torsion free subsheaf of rank $1$ of the constant sheaf of rational functions $\mathcal{K}$ on $C$. If $h^0(\vv)\geq g$ and $\deg(\vv)=2g-2$ then there exists an embedding of the dualizing sheaf  $\ww\hookrightarrow\mathcal{K}$ whose image is $\vv$. In fact, this can be seen, for instance, rephrasing \cite[p. 110 top]{S} in terms of sheaves. So consider the subsheaf of $\mathcal{K}$ defined by
$$
\vv:=\oo_C\langle 1,t^{b_1},\ldots,t^{b_{\delta_P-1}},t^{c_1},\ldots,t^{c_{\delta_Q}}\rangle
$$
that is, it is generated by the (global) sections $1,t^{b_i},t^{c_j}\in k(t)=k(C)$ defined by the statement of the theorem. Since the $b_i$ and $c_j$ all differ to one another, these sections are linear independent, so $h^0(\vv)\geq \delta_P+\delta_Q=g$. We claim that $\deg(\vv)=2g-2$ as well.

In order to prove this, first note that $1\in H^0(\vv)$ thus $\vv\supset\oo$ and hence
$$
\deg(\vv)=\sum_{R\in C}\dim(\vv_R/\oo_R).
$$
Now if $R\neq P,Q$, then, clearly, $\vv_R=\oo_R$ and any such an $R$ gives no contribution to the degree. Let us compute the dimension for the point $Q$. Recall, for instance, from \cite[Prp. 1]{M} that
\begin{equation}
\label{equdwq}
\dim(\vv_Q/\oo_Q)=\#(v_Q(\vv_Q)\setminus\sss_Q)
\end{equation}
We assert that
\begin{equation}
\label{equvwq}
v_Q(\vv_Q)\setminus\sss_Q=-(\gamma_P+G_Q)\,\bigcup\,\{-\gamma_P+1,-\gamma_P+2,\ldots,-1\}\,\bigcup\, G_Q
\end{equation}
Indeed, to prove ``$\supset$" note that $\gamma_P+G_Q=\{c_i\}_{i=1}^{\gamma_P}$, and that $t^{c_i}\in\vv_Q$ by construction; on the other hand, $v_Q(t^{c_i})=-c_i$ so the inclusion of the first set in the above union holds. Now let $r\in\zz$ be such that $r\leq\gamma_P-1$; then, in particular, $c_i-r\geq 2$ for any $i$; since the number of $c_i$'s is the number of gaps of $\sss_Q$ and $1$ is a gap, we have that $c_j-r\in\sss_Q$ for some $j$, so take $f\in\oo_Q$ for which $v_Q(f)=c_j-r$; thus $t^{c_j}f\in\vv_Q$ and $v_Q(t^{c_j}f)=-c_j+c_j-r=-r$ so the inclusions of the second and third sets above hold as well.  

To prove ``$\subset$" we use the fact that $C$ is monomial and, in particular, so is the local ring of $Q$. That is, if one writes
\begin{equation}
\label{equun3}
\sss_Q=\{0,s_1,\ldots,s_n,\beta_Q,\rightarrow\}
\end{equation}
then we have that
\begin{equation}
\label{equooq}
\oo_Q=k\oplus kt^{-s_1}\oplus\ldots\oplus kt^{-s_n}\oplus\mathcal{C}_Q
\end{equation}
since $t^{-1}$ is a local parameter for  $\overline{Q}$. From this we conclude that any $a\in v_Q(\vv_Q)$ is of the form $a=v_Q(t^dt^{-s})=s-d$ where $d\in\{ 0,b_2,\ldots,b_{\delta_P},c_1,\ldots,c_{\delta_Q}\}$ and $s\in\sss_Q$. So it suffices to show that if $s-d\leq -\gamma_P$ then $d-s-\gamma_P\not\in S_Q$. Now the inequality implies that $d\in\{c_1,\ldots,c_{\delta_Q}\}=\gamma_P+G_Q$. Hence there is an $\ell\in G_P$ for which $d=\gamma_P+\ell$ yielding that $d-s-\gamma_P=\ell-s$ which cannot be in $\sss_Q$ otherwise so would be $\ell$. Thus ``$\subset$" is proved too.

From (\ref{equdwq}) and (\ref{equvwq}) we conclude that
\begin{equation}
\label{equdmq}
\dim(\vv_Q/\oo_Q)=2\delta_Q+\beta_P-2
\end{equation}

Now let us compute the dimension for the point $P$. We claim that
\begin{equation}
\label{equwpl}
\vv_P=k\oplus kt^{b_2}\oplus\ldots\oplus kt^{b_{\delta_P}}\oplus\cp
\end{equation}
Indeed, ``$\supset$" trivially comes from the fact that 
\begin{equation}
\label{equwwp}
\vv_P=\oo_P+t^{b_2}\oo_P+\ldots+t^{b_{\delta_P}}\oo_P
\end{equation}
To prove  ``$\subset$" note first that 
$$
v_P(k\oplus kt^{b_2}\oplus\ldots\oplus kt^{b_{\delta_P}}\oplus\cp)=\kk_P
$$
where the latter was defined in (\ref{equkkp}). Since ``$\supset$" holds, it is enough showing that 
\begin{equation}
\label{equvwp}
v_P(\vv_P)=\kk_P
\end{equation}
as well. To see so, recall again that $C$ is monomial thus $\oo_P$ also satisfies (\ref{equooq}) replacing $Q$ by $P$ and $t^{-1}$ by $t$. Combining this with (\ref{equwwp}) it suffices to prove that $\sss_P+\kk_P\subset \kk_P$ to get (\ref{equvwp}); or, in the language of semigroups, that $\kk_P$ is an $\sss_P$ relative ideal; but this is known, for instance, by \cite{J}, so ``$\subset$" holds and the claim follows.

Thus, from (\ref{equwpl}), we have that $\dim(\vv_P/\cp)=\delta_P$.  Therefore one may write
\begin{align*}
\dim(\vv_P/\op)&=\dim(\obp/\oo_P)+\dim(\vv_P/\cp)-\dim(\obp/\cp) \\
          &=2\delta_P-\beta_P
\end{align*}
Combining this along with (\ref{equdmq}) yields
\begin{align*}
\deg(\vv)&=\dim(\vv_P/\oo_P)+\dim(\vv_Q/\oo_Q)\\
          &=2\delta_P-\beta_P+2\delta_Q+\beta_P-2\\
          &=2g-2
\end{align*}
So $\vv$ is an isomorphic image of $\ww$ in $\mathcal{K}$ as desired. Call 
$$
V:=k\oplus kt^{b_2}\oplus\ldots\oplus kt^{b_{\delta_P}}\oplus kt^{c_1}\oplus\ldots\oplus kt^{c_{\delta_Q}}
$$
Since $\vv\cong\ww$, in particular, $h^0(\vv)=g$ and hence $H^0(\vv)=V$. Now $\cb=\pum$ so the canonical model $C'$ is the image of the morphism $\kappa :\pum\to\mathbb{P}^{g-1}$ defined by the linear system $\mathcal{L}(\oo_{\cb}\vv,H^0(\vv))=\mathcal{L}(\oo_{\pum}\langle V\rangle,V)$, that is, $C'=\kappa(\pum)$ where
\begin{gather*}
\begin{matrix}
\kappa : & \mathbb{P}^1 & \longrightarrow & \mathbb{P}^{g-1} \\
               & (1:t)              & \longmapsto     & (1:t^{b_2}:\ldots: t^{b_{\delta_P}}:t^{c_1}:\ldots :t^{c_{\delta_Q}})
\end{matrix}
\end{gather*}
and we are done.
\end{proof}

Now we recall a result for general scrolls which will be helpful for our purposes. We just warn the reader that the statement -- as any involving monomiality -- depends on a choice of coordinates of the ambient space, though one is allowed at least to reodering them. 

\begin{lema}
\label{Danielle}
A rational monomial curve $(1:t^{a_1}:\ldots :t^{a_{N}})\subset \pp^{N}$ lies on a $d$-fold scroll $S_{m_1m_2\ldots m_d}$ if and only if there is a partition of the set $\{0=a_0,a_1,\ldots,a_{N}\}$ into $d$ subsets, with, respectively, $m_1+1, m_2+1,\ldots,m_d+1$ elements, such that the elements of all of these subsets can be reordered within an arithmetic progression with the same common difference.
\end{lema}

\begin{proof}
See \cite[Lem 3.2]{LMM}
\end{proof}

With this in mind, we are ready to prove the following result.

\begin{teo} 
\label{thmrat}
Let $C$ be a rational monomial curve of genus at least $1$. Then $\gon(C)=d$ if and only if its canonical model $C^{\prime}$ lies on a $(d-1)$-fold scroll, and does not lie on any scroll of smaller dimension. 
\end{teo}

\begin{proof}
We will proceed by induction on $d$. Since $g\geq 1$, then $\gon(C)\geq 2$. So first note that the statement of the theorem for $d=2$ corresponds to the following sentence: $C$ has gonality $2$ if and only if $C'$ is the rational normal curve of degree $g-1$ in $\mathbb{P}^{g-1}$. But this holds from \cite[Thm. 3.4]{KM} and \cite[Thm. 2.1]{Mt}.

For the remainder, write
$$
C^{\prime} = (1:t^{b_2}:\cdots:t^{b_{\delta_P}}:t^{c_1}:\cdots:t^{c_{\delta_Q}})\subset\proje^{g-1}
$$
as in Theorem \ref{prprat} and set
$$
A:=\{0,b_2,\ldots, b_{\delta_P},c_1,\ldots, c_{\delta_Q}\}
$$

Now we prove the result for a general $d$ assuming that it holds for any smaller integer. To check sufficiency, suppose $C^{\prime}$ lies on a $(d-1)$-fold scroll. By Lemma \ref{Danielle}, there is a partition of $A$ into $d-1$ subsets, say $A_1$, $\ldots$, $A_{d-1}$, all forming an arithmetic progression with the same common difference, say $r$. 
Write
\begin{align*}
A_1&=\{0, r, \ldots, e_1\}\\
A_2&=\{a_2, a_2 +r, \ldots, e_2\}\\
      &\ \,  \vdots\\
A_{d-1}&=\{a_{d-1}, a_{d-1}+r, \ldots, e_{d-1}\}
\end{align*}
Supposing also that $d-1$ is the smallest dimension of a scroll on which $C'$ can lie, we see from Lemma \ref{Danielle}  that 
 any ending element $e=e_i$ of $A_i$ satisfies $e+r\notin A$.

Now consider the subsheaf of $\mathcal{K}$ on $C$ defined by
$$
{\cal F}:={\cal O}_C\left\langle 1, t^r\right\rangle
$$
generated by the (global) sections $1,t^r\in k(C)=k(t)$. We claim that 
$$
\deg(\aaa)=d
$$
In order to prove so, note that since $C$ is monomial and $t$ (resp. $t^{-1}$) is a local parameter for $P$ (resp. $Q$) we have that
$$
v_R(\aaa_R)=
\begin{cases}
\sss_P\,\bigcup\, (\sss_P+r) & \text{if}\ R=P\\
\nn & \text{if}\ R\neq P,Q \\
\sss_Q\,\bigcup\,(\sss_Q-r) & \text{if}\ R=Q
\end{cases}
$$
Thus, from what was said in the proof of the prior theorem, we conclude that
$$
\deg(\aaa)=\#(\sss_P+r\setminus\sss_P)+\#(\sss_Q-r\setminus\sss_Q)
$$

Now let $e=e_i\in A_i$ for some $i\in\{1,\ldots,d-1\}$. Consider the following cases:

\noindent (i) if  $e=\gamma_P-\ell$ with $\ell\in G_P$ and $\ell\geq r$ then $\ell-r\in S_P$ since $e+r\not\in A$, which implies that $\ell\in\sss_P+r\setminus\sss_P$;

\noindent (ii) if  $e=\gamma_P-\ell$ with $\ell\in G_P$ and $\ell<r$ then $r-\ell\in S_Q$ since $e+r\not\in A$, which implies that $-\ell\in\sss_Q-r\setminus\sss_Q$;

\noindent (iii) if  $e=\gamma_P+\ell$ with $\ell\in G_Q$ then $\ell+r\in S_Q$ since $e+r\not\in A$, thus $\ell\in\sss_Q-r\setminus\sss_Q$;

Moreover, $-r\in \sss_Q-r\setminus\sss_Q$. Therefore, combining this along with the three statements above we get that $\deg(\aaa)\geq d$. Let us now prove that equality holds.

If $s\in \sss_P$ with $s+r\notin S_P$, then $\gamma_P-(s+r)\in A$ and $\gamma_P-s \notin A$; therefore 
$\gamma_P-(s+r)=e_i$ for some $i$. 

If $s\in S_Q$ with $s-r\notin S_Q$ we have two cases to analyze. First,  if $s>r$, then $\gamma_P+(s-r)\in A$ and $\gamma_P+s\not\in A$; therefore $\gamma_P+(s-r)=e_i$ for some $i$. 

Otherwise, if $s<r$, we break down this case within two new ones.  If $r-s\not\in\sss_P$, then $\gamma_P-(r-s)\in A$ and $\gamma_P+s\notin A$  therefore $\gamma_P-(r-s)=e_i$ for some $i$. 

If not, i.e., if $r-s\in\sss_P$, take a new partition of $A$ with subsets all forming an arithmetic progression with same common difference $r-s$.  Note that one can always do that no matter if no elements of $A$ will be linked to one another. But we claim that, in our case, the new partition splits $A$ into $d-1$ disjoint sets as well. In fact, if $\gamma_P-\ell\in A$ for $\ell\in G_P$ with $\ell\geq r-s$ then $\gamma_P-l+r-s \in A$ because otherwise $l-r+s\in\sss_P$ which implies that  $\ell\in\sss_P$ which cannot occur. So if $\gamma_P-\ell$,  with $\ell\in G_P$, ends a subset of $A$ in the new partition then $\ell < r-s$ and $\gamma_P+r-s-\ell\not\in A$; thus $r-s-\ell\in\sss_Q$ and hence $r-\ell\in\sss_Q$. But this implies that $\gamma_P-\ell$ ends a subset of $A$ in the first partition as well. And if $\gamma_P+\ell$, with $\ell\in G_Q$, ends a subset of $A$ in the new partition  then $\gamma_P+\ell+r-s\not\in A$, which implies that $\ell+r-s\in\sss_Q$ and so $\ell+r\in\sss_Q$; thus $\gamma_P-\ell$ also ends a subset of $A$ in the first partition. It follows that if the new partion subdivide $A$ in $d'$ subsets, then $d'\leq d-1$. But equality should hold otherwise, by Lemma \ref{Danielle}, $C'$ lies on a $d'$-fold scroll with $d'<d-1$, which contradicts the hypothesys. So one may replace $r$ by $r-s$ and restart the proof over and over again until we make sure that $r-s\not\in\sss_P$ for any $s\in\sss_Q$ with $s-r\not\in\sss_Q$. This will surely happen as $r$ decreases each step. 

So $\deg({\cal F})=d$ and $\gon(C)\leq d$. But it has to be $d$, because otherwise $C^{\prime}$ would lie on a scroll of dimension smaller than $d-1$ due to our induction hypothesis.

Conversely, necessity follows from Theorem \ref{thmscr}.(I) and induction.
\end{proof}


\begin{thebibliography}{BDF2}
\bibitem{AM} A. Andreotti, A. L. Mayer, \emph{On Period Relations for Abelian Integrals on Algebraic Curves}, Annali della Scuolla Normale Superiore di Piza, 21 2 (1967) 189-238
\bibitem{B} D. W. Babbage, 
{\it A note on the quadrics through a canonical curve}, J. Lodon Math. Soc. 14 (1939), 310--315.
\bibitem{Br} H. Bresinsky, 
{\it Monomial Space Curves in $\mathbb{A}^3$ as Set-Theoretic Complete Intersections}, Proceedings of the American Mathematical Society 75 (1979), 23--24.
\bibitem{BF}     V. Barucci, R. Fr\"oberg, {\it One-Dimensional Almost Gorenstein Rings}, Journal of Algebra 188 (1997), 418--442.
\bibitem{BS} M. Brundu, G. Sacchiero, {\it Stratification of the moduli space of fourgonal curves}, Proc. Edinb. Math. Soc., Vol. 57, Issue 03, 631 - 686 (2014)
\bibitem{CE} G. Casnati, T. Ekedahl, {\it Covers of algebraic varieties. I. A general structure theorem, covers of degree 3,4 and Enriques
surfaces}, J. Algebraic Geom. 5 (1996), no. 3, 439 - 460.
\bibitem{CS} A. Contiero and K.-O. Stoehr. {\it Upper bounds for the
dimension of moduli spaces of curves with symmetric Weierstrass
semigroups}, J. Lond. Math. Soc., 88 (2013) 580--598.
\bibitem{Cp}  M. Coppens, \emph{Free Linear Systems on Integral Gorenstein Curves}, Journal of Algebra 145 (1992), 209--218.
\bibitem{CFM1} E. Cotterill, L. Feital, R. V. Martins,
\emph{Dimension Counts for Singular Rational Curves via Semigroups},
ArXiv:1511.08515v2
\bibitem{CFM2} E. Cotterill, L. Feital, R. V. Martins,
\emph{Singular Rational Curves with Points of Nearly-Maximal Weight},
Journal of Pure and Applied Algebra, doi 10.1016/j.jpaa.2017.12.017.

\bibitem{EHKS}        D. Eisenbud, J. Harris, J. Koh, M. Stillmann, {\it Determinantal Equations for Curves of High Degree}, American Journal of Mathematics 110 (1988) 513-539
\bibitem{EH} D. Eisenbud, J. Harris, \emph{On Varieties of Minimal Degree}, Proceedings of Symposia in Pure Mathematics 46 (1987), 3--13
\bibitem{En} F. Enriques, {\it Sulle curve canoniche di genera $p$ cello spazio a $p-1$ dimensioni},
Rend. Accad. Sci. Ist. Bologna, 23 (1919), 80--82.
\bibitem{Hr} J. Herzog, {\it Generators and Relations of Abelian Semigroups and Semigroup Rings}, Manuscripta
Math. 3 (1970), 175-193
\bibitem{HU} J. Hotchkiss, B. Ullery, {\it The Gonality of Complete Intersection Curves}, arXiv:1706.08169
\bibitem{J} J. J\"ager, \emph{L\"angeberechnungen und Kanonische Ideale in Eindimensionalen Ringen}, Arch. Math. 29 (1977) 504-512
\bibitem{KM} S. L. Kleiman, R. V. Martins, {\it The Canonical Model of a Singular Curvee}, Geometria Dedicata. 139 (2009), 139-166.
\bibitem{LMM} D. Lara, S. Marchesi, R. V. Martins, \emph{Curves with Canonical Models on Scrolls}, International Journal of Mathematics, Vol. 27, No. 5 (2016) 1650045-1-30.
\bibitem{Mt}   R. V. Martins, \emph{On Trigonal Non-Gorenstein Curves with Zero Maroni Invariant}, Journal of Algebra 275 (2004), 453--470.
\bibitem{M2}   R. V. Martins, {\it Trigonal non-Gorenstein curves}, J. Pure Appl. Algebra, 209 (2007), 873--882.
\bibitem{M} T. Matsuoka, \emph{On the degree of singularity of one-dimensional analytically irreducible noetherian rings}, J. Math. Kyoto Univ. 11 (1971) 485-491
\bibitem{Mr} R. M. Mir\'o-Roig, \emph{The Representation Type of Rational Normal Scrolls}, Rend. Circ. Mat. Palermo 62 (2012), 153--164.
\bibitem{Rd}  M. Reid, \emph{Chapters on Algebraic Surfaces}, arXiv:alg-geom/9602006v1 6 Feb 1996; Lectures of a summer programm Park City, UT, 1993
\bibitem{RS}      R. Rosa, K.-O. St\"ohr, {\it Trigonal Gorenstein Curves}, Journal of Pure and Applied Algebra 174 (2002), 187-205.
\bibitem{R}     M. Rosenlicht, {\it Equivalence Relations on Algebraic Curves}, Annals of Mathematics 56 (1952), 169--191
\bibitem{Sc} F.-O. Schreyer, {\it Syzygies of Canonical Curves and Special Linear Series}, Mathematische Annalen 275 (1986), 105--137.
\bibitem{S}        K.-O. St\"ohr, {\it On the Poles of Regular Differentials of Singular Curves}, Boletim da Sociedade Brasileira de Matem\'atica 24 (1993), 105--135.
\end{thebibliography}
\end{document}